\documentclass{article}
\usepackage{a4wide}
\usepackage{graphicx}
\usepackage{color}
\usepackage{amscd}
\usepackage{framed}
\usepackage{bbm}
\usepackage{amsmath,amsthm,amssymb}
\usepackage{fancyhdr,a4wide}
\usepackage{comment}
\newcommand{\mc}{\mathcal}
\newcommand{\mbb}{\mathbb}
\newcommand{\mr}{\mathrm}
\newcommand{\argmin}{\mathop{\rm argmin}\limits}

\newcommand{\veca}{\mathbf{a}}

\newcommand{\vecv}{\mathbf{v}}

\newcommand{\vecx}{\mathbf{x}}
\newcommand{\vecX}{\mathbf{X}}

\newcommand{\vecbeta}{\boldsymbol \beta}

\newcommand{\vecdelta}{\boldsymbol \delta}

\usepackage{bm}

\usepackage{algorithm}
\usepackage{algorithmic}

\numberwithin{equation}{section}
\newtheorem{theorem}{Theorem}[section]
\newtheorem{proposition}{Proposition}[section]

\newtheorem{lemma}{Lemma}[section]

\newtheorem{definition}{Definition}[section]
\newtheorem{assumption}{Assumption}[section]

\begin{document}
\title{Estimation of sparse linear regression coefficients under $L$-subexponential covariates}
\author{Takeyuki Sasai
\thanks{Department of Statistical Science, The Graduate University for Advanced Studies, SOKENDAI, Tokyo, Japan. Email: sasai@ism.ac.jp}
}
\maketitle
\begin{abstract}
	We tackle estimating sparse coefficients in a linear regression when
	the covariates are sampled from an $L$-subexponential random vector. This
	vector belongs to a class of distributions that exhibit heavier
	tails than Gaussian random vector. Previous studies have established
	error bounds similar to those derived for Gaussian random vectors. 
	However, these  methods require stronger conditions than those used for
	Gaussian random vectors to derive the error bounds. In this 
	study, we present an error bound identical to the one
	obtained for Gaussian random vectors up to constant factors without
	imposing stronger conditions,  when the covariates are drawn from an
	$L$-subexponential random vector. Interestingly, 
	we employ an $\ell_1$-penalized Huber regression, 
	which is known for its robustness against heavy-tailed random noises
	rather than covariates. We believe that this study uncovers a
	new aspect of the  $\ell_1$-penalized Huber regression method.
\end{abstract}

\section{Introduction}
\label{intro}
In this study, we consider sparse  linear regression. We define a sparse linear regression model as follows:
\begin{align}
	\label{model:normal}
	y_i = \vecx_i^\top\vecbeta^*+\xi_i,\quad  i=1,\cdots,n,
\end{align}
where $\vecbeta^* \in \mbb{R}^d$ is the true coefficient vector, $\left\{\vecx_i\right\}_{i=1}^n$ is a sequence of independent and identically distributed (i.i.d.) random vectors and $\left\{\xi_i\right\}_{i=1}^n$ is a sequence of i.i.d. random variables. Throughout this study, we assume that the number of  nonzero elements of $\vecbeta^*$ is $s(\leq d)$, and that $\{\vecx_i\}_{i=1}^n$ is independent of  $\{\xi_i\}_{i=1}^n$, and $d/s\geq 3$.

For a vector $\vecv \in \mbb{R}^d$, we define its $\ell_2$ norm as $\|\vecv\|_2$.
Assume that  $\{\vecx_i\}_{i=1}^n$ and $\{\xi_i\}_{i=1}^n$ are  sequences of i.i.d. random covariate vectors sampled from a multivariate Gaussian distribution with $\mbb{E}\vecx_i = 0$ and $\mbb{E}\vecx_i \vecx_i^\top = I$, and random noises sampled from a Gaussian distribution with $\mbb{E}\xi_i = 0$ and $\mbb{E}\xi_i^2 = 1$, respectively. 
We define $\lesssim$ as an inequality up to a numerical factor.
It is known that, we can construct an estimator $\hat{\vecbeta}$  such that
\begin{align}
\label{ine:bound}
	\mbb{P}\left(\|\hat{\vecbeta} -\vecbeta^*\|_2 \lesssim \sqrt{\frac{s\log (d/s)}{n}}+\sqrt{\frac{\log (1/\delta)}{n}}\right)\geq 1-\delta,
\end{align}
when $\sqrt{\frac{s\log (d/s)}{n}}+\sqrt{\frac{\log (1/\delta)}{n}}$ is sufficiently small \cite{BelLecTsy2018Slope}.
Since the invention of the lasso \cite{Tib1996Regression}, numerous studies have been conducted on the estimation of $\vecbeta^*$ from \eqref{model:normal}.
In many studies, the sequences of covariates $\{\vecx_i\}_{i=1}^n$ are sampled from a Gaussian or an $L$-subGaussian random vector (see Definition \ref{d:sgrvec}), and the noise $\{\xi_i\}_{i=1}^n$ is drawn from a Gaussian or  other heavy-tailed distributions, leading to results similar to  in \eqref{ine:bound}.
However, only a few studies considered the case when $\{\vecx_i\}_{i=1}^n$ is drawn from heavy-tailed random vectors other than Gaussian or $L$-subGaussian random vectors \cite{SivBenRav2015Beyond, LecMen2017Sparse, FanWanZhu2021Shrinkage, KucCha2022Moving,GenKip2022Generic},  resulting in  that are somewhat inferior results compared to those of \eqref{ine:bound}.
In this study, we make the assumption that $\{\vecx_i\}_{i=1}^n$ is drawn from an $L$-subexponential random vectors (see Definition \ref{d:sgrvec}), which have heavier tails than $L$-subGaussian random vector, and we derive a result that is comparable to \eqref{ine:bound}, except for the numerical constant and the dependency of $L$.

In the rest of this paper.
the value of the numerical constant $C$ be allowed to change from line to line.
\section{Related work, our estimation method and key idea}
\subsection{Some definitions}
\label{label:somedefinitions}
In Section \ref{label:somedefinitions}, we introduce the definitions which are used later.
\begin{definition}[$\psi_\alpha$-norm]
	\label{d:orlicz}
		For a random variable $f$, we let
		\begin{align}
			\|f\|_{\psi_\alpha}:=	\inf\left\{ \eta>0\,:\, \mbb{E}\exp\left|f/\eta\right|^\alpha\leq 2\right\} < \infty.
		\end{align}	
	\end{definition}

\begin{definition}[$L$-subGaussian and $L$-subexponential random vector]
	\label{d:sgrvec}
		Let $L$ be a numerical constant.
		A random vector $\vecx \in \mbb{R}^d$  is said to be an $L$-subGaussian and $L$-subexponential random vector if for any fixed $\vecv \in \mbb{R}^d$,
		\begin{align}
		\label{ine:sgrvec}
			\|\langle \vecx -\mbb{E}
			\vecx,\vecv\rangle\|_{\psi_2}\leq L\left(\mbb{E}\langle \vecx -\mbb{E}
			\vecx,\vecv\rangle^2\right)^\frac{1}{2}\text{ and }			\|\langle \vecx-\mbb{E}\vecx,\vecv\rangle\|_{\psi_1}\leq L\left(\mbb{E}\langle \vecx -\mbb{E}
			\vecx,\vecv\rangle^2\right)^\frac{1}{2},
		\end{align}
		respectively.
\end{definition}
\begin{definition}[$L$-subGaussian and $\sigma$-subexponential  random variable]
	\label{d:sgrvar}
		Let $\sigma$ be a numerical constant
		A random variable $\xi $ is said to be a $\sigma$-subGaussian and  $\sigma$-subexponential random variable if 
		\begin{align}
		\label{ine:sgrvar}
			\| \xi-\mbb{E}\xi \|_{\psi_2}\leq \sigma\text{ and }\|\xi-\mbb{E}\xi\|_{\psi_1}\leq \sigma,
		\end{align}
		respectively.
\end{definition}
We observe that $\|f-\mbb{E}f\|_{\psi_1}\leq \|f-\mbb{E}f\|_{\psi_2}$ holds for a $\sigma$-subGaussian random variable $f$ (as stated in  Lemma 2.7.7 of \cite{Ver2018High}), and we see that an $L$-subGaussian random vector is also an $L$-subexponential random vector.

\subsection{Related work}
\label{sec:rel}
For a vector $\vecv \in \mbb{R}^d$,  its $\ell_1$ norm defined as $\|\vecv\|_1$.
The estimation problem of $\vecbeta^*$ in \eqref{model:normal} from the data from $\{\vecx_i,y_i\}_{i=1}^n$ is considered in 
\cite{GenKip2022Generic} and \cite{SivBenRav2015Beyond},   where  $\{\vecx_i\}_{i=1}^n$ and $\{\xi_i\}_{i=1}^n$ are drawn from isotropic $L$-subexponential random vectors and $\sigma$-subexponential random variables, respectively.
In \cite{SivBenRav2015Beyond}, the estimation method used  is the $\ell_1$-penalized least squares:
\begin{align}
	\hat{\vecbeta} = \argmin_{\vecbeta \in \mbb{R}^d} \left\{ \frac{1}{n}\sum_{i=1}^n(y_i-\vecx_i^\top \vecbeta)^2 +\lambda_s \|\vecbeta\|_1\right\},
\end{align}
where $\lambda_s$ is a tuning parameter, and the estimation method in \cite{GenKip2022Generic} is  $\ell_1$-constrained least squares.
We define $\lesssim$ as the inequality up to   constant factor.
The methods in \cite{GenKip2022Generic} and \cite{SivBenRav2015Beyond} yield  results similar to those of \eqref{ine:bound}.
However, to obtain their errors bounds, they require that $n$ is at least larger than $s (\log (d/s))^2$, not  $s \log (d/s)$ to get their error bounds due to the heavy-tailedness of $\{\vecx_i\}_{i=1}^n$. 
In Section \ref{label:o}, we explain how the  heavy-tailedness of $\{\vecx_i\}_{i=1}^n$ negatively affects the results, and  propose a method to eliminate the additional condition.

\subsubsection{Other related work}
\label{sec:relo}
\cite{FanWanZhu2021Shrinkage} considered a case where the covariates are drawn from a heavy-tailed random vector $\vecx$ for which,  for any $\vecv \in \mbb{S}^{d-1}$, $\mbb{E}\langle \vecx,\vecv\rangle^4\leq K^4$, where $K$ is some constant holds. The method proposed by \cite{FanWanZhu2021Shrinkage} requires additional conditions to be satisfied, namely, that $(s^2 \log d )/ n$ and $\|\vecbeta^*\|_1 \sqrt{(\log d)/n}$ are sufficiently small.
\cite{LecMen2017Sparse} considered  the restricted eigenvalue condition for the design matrix when the covariates are drawn from a heavy-tailed random vector $\vecx$ for which,  for any $\vecv \in \mbb{S}^{d-1}$, and $p$ such that $4\leq p \lesssim  \log d$,  $\mbb{E}|\langle \vecx,\vecv\rangle |^p\leq L^p$,  holds.
Moving on to \cite{KucCha2022Moving}, they investigated  a scenario where  the covariates are drawn from a marginal $L$-subWeibull random vector. This analysis was performed under the condition that $n$ is greater  than $s(\log (d/s))^2$.
A marginal $L$-subWeibull random vector $\vecx = (x_1,\cdots, x_d) \in \mbb{R}^d$ is a random vector for which, for  $0<\alpha$ and $L$, each coordinate satisfies $\|x_i\|_{\psi_\alpha} \leq L$. From the definition of $\|\cdot\|_{\psi_\alpha}$, we see that  $\sigma$-subexponential random variable is a $\sigma$-subWeibull random variable with $0< \alpha\leq 1$  and the latter has heavier tail than the former. Hence, \cite{KucCha2022Moving}  dealt with  more generalized situations than those considered in the present study.
Validating the effectiveness of the observations made in Section \ref{label:o} regarding these situations will be a future task.
\subsection{Our estimation method and key idea}
\label{label:o}
Define the Huber loss and its derivative as 
\begin{align}
	H(t) = \begin{cases}
	|t| -1/2 & (|t| > 1) \\
	t^2/2 & (|t| \leq 1)
	\end{cases} \text{ and }
	h(t) =	\frac{d}{dt} H(t) = \begin{cases}
		\mr{sgn}(t)\quad &(|t| >1)\\
		t\quad &(|t| \leq 1)
		\end{cases},
\end{align}
respectively.
The estimator is as follows:
\begin{align}
	\label{ine:Huber}
	\hat{\vecbeta} = \argmin_{\vecbeta \in \mbb{R}^d} \left\{ \sum_{i=1}^n \lambda_o^2 H\left(\frac{y_i-\vecx_i^\top \vecbeta}{\lambda_o \sqrt{n}}\right) +\lambda_s \|\vecbeta\|_1\right\},
\end{align}
where $\lambda_o$ is the tuning parameter for  Huber loss.
The $\ell_1$-penalized Huber loss is known to be effective when $\{\xi_i\}_{i=1}^n$ is drawn from heavy-tailed distributions or is affected by outliers \cite{SheOwe2011Outlier,NguTra2012Robust,DalTho2019Outlier,SunZhoFan2020Adaptive}. 
In this study, we find that the $\ell_1$-penalized Huber loss remains effective  when $\{\vecx_i\}_{i=1}^n$ is drawn from an $L$-subexponential random vector, and we can derive an error bound that is equivalent to \eqref{ine:bound}, except for  constant factors.
We note that, an estimator which is essentially identical to \eqref{ine:Huber} was also studied in detail in \cite{SunZhoFan2020Adaptive}, and derived results similar to ours. The main interest of \cite{SunZhoFan2020Adaptive} is the robustness of  \eqref{ine:Huber} for the heavy-tailedness of $\{\xi\}_{i=1}^n$, and \cite{SunZhoFan2020Adaptive} assumes that $\{\vecx_i\}_{i=1}^n$ is $L$-subGaussian. Thus, this study and \cite{SunZhoFan2020Adaptive} have different main interests.

Let us elucidate why the squared loss is ineffective in handling the heavy-tailed characteristics of $\{\vecx_i\}_{i=1}^n$.
Intuitively, the Huber loss limits the range of movement for $\{\xi_i\}_{i=1}^n$, thereby creating space for greater activity in $\{\vecx_i\}_{i=1}^n$.

Here we explain why the squared loss is not effective to the heavy-tailedness of $\{\vecx_i\}_{i=1}^n$ and $\{\xi_i\}_{i=1}^n$. 
For a vector $\vecv \in \mbb{R}^d$, we define its $\ell_0$ norm as $\|\vecv\|_0$ representing the number of  non-zero elements in $\vecv$. Additionally,  then, for $l = 0, 1,2$, we define $d$-dimensional $\ell_l$-ball  with diameter $a$ as 
$a\mbb{B}^d_l = \{\vecv \in \mbb{R}^d \,\mid\, \|\vecv\|_l\leq a \}$. 
When  deriving error bounds similar to  \eqref{ine:bound} with squared loss,
it becomes crucial  to evaluate the following quantity:
\begin{align}
	\label{ine:plsrabove}
	\sup_{\vecv \in  r_1\mbb{B}^d_1 \cap r_2 \mbb{B}^d_2}\sum_{i=1}^n \frac{1}{n}\xi_i \vecx_i^\top \vecv.
\end{align}
When  $\vecx_i$ is  an $L$-subGaussian random vector and  $\xi_i$ is a $\sigma$-subGaussian distribution, the cross term $\xi_i \vecx_i$ becomes a  $CL\sigma$-subexponential random vector (Lemma 2.7.7. of \cite{Ver2018High}).
However,  if  $\vecx_i$ is  an $L$-subexponential random vector, the cross term $\xi_i \vecx_i$  generally ceases to be  a  $CL\sigma$-subexponential random vector even when   $\xi_i$ is a Gaussian.	Consequently, obtaining the concentration inequality in \eqref{ine:plsrabove}, similar to the case where $\vecx_i$ is an $L$-subGaussian vector, becomes challenging.

In the context of $\ell_1$-penalized Huber regression, there is no need  to evaluate \eqref{ine:plsrabove}.
Alternatively, it is sufficient to evaluate  the following:
\begin{align}
	\label{ine:phsrabove}
	\sup_{\vecv \in  r_1\mbb{B}^d_1 \cap r_2 \mbb{B}^d_2}\sum_{i=1}^n \frac{1}{n}h\left(\frac{\xi_i}{\lambda_o\sqrt{n}}\right) \vecx_i^\top \vecv.
\end{align}
When  $\vecx_i$ is  an $L$-subexponential random vector and  $\xi_i$ is a $\sigma$-subexponential random variable, the cross term $h\left(\frac{\xi_i}{\lambda_o\sqrt{n}}\right)\vecx_i$ remains an  $L$-subexponential random vector because $h(\cdot)$ is bounded. Based to this observation, we succeed in obtaining a result similar to that of \eqref{ine:bound} up to constant factors.

\section{Our main result and key propositions}
\label{sec:mainproof}
In Section \ref{sec:mainproof}, we give our main result (Theorem \ref{t:main}) and Propositions \ref{p:main:upper} and \ref{p:main:lower}, which are used in the proof of Theorem \ref{t:main}. The proofs of Theorem \ref{t:main} and Propositions \ref{p:main:upper} and \ref{p:main:lower} are given in the Appendix.
Before we state our main result,  we introduce the following assumption:
\begin{assumption}
	\label{a:intro}
		Assume that
		\begin{itemize}
			\item [(i)]$\left\{\vecx_i \right\}_{i=1}^n$ is a sequence of i.i.d. random vectors sampled from an $L$-subexponential random vector with $\mbb{E}\vecx_i = 0,\,\mbb{E}\vecx_i \vecx_i^\top = I$,  
			\item [(ii)] $\left\{\xi_i\right\}_{i=1}^n$ is a sequence of i.i.d. random variables with  $\mbb{E}\xi_i^2 \leq \sigma^2$.
		\end{itemize}
	\end{assumption}
We note that the condition (ii) in Assumption \ref{a:intro} is weaker than  $\sigma$-subexponential property of $\{\xi_i\}_{i=1}^n$. 
Define
\begin{align}
	r_{d,s} = \sqrt{\frac{s\log(d/s)}{n}},\quad r_\delta = \sqrt{\frac{\log(1/\delta)}{n}}.
\end{align}
Then, we state our main theorem:
\begin{theorem}
\label{t:main}
Suppose that Assumption \ref{a:intro} holds. Let $c_1$ and $c_2$ be numerical constants, which are defined in Propositions \ref{p:main:upper} and \ref{p:main:lower}.
For the tuning parameters $\lambda_o$ and $\lambda_s$, we assume that  
\begin{align}
	\lambda_o \sqrt{n}&  \geq 576\sigma L^2 ,\quad  \lambda_s \sqrt{s}= c_s \lambda_o \sqrt{n} L(r_{d,s}+r_\delta),
\end{align}
where $c_s$ denotes a numerical constant such that $c_s \geq 5c_1$. 
Assume that $r_1 =3\sqrt{s}r_2$ and $r_2= 5L\lambda_o \sqrt{n}(c_1+c_2+c_s)   (r_{d,s}+r_\delta)$, and $n$ is sufficiently large so that 
\begin{align}
	 1\geq \max\left\{r_{d,s}, r_\delta, r_2,320 L^4 (c_1+c_2+c_s)  (r_{d,s}+r_\delta) \right\}.
\end{align}
Then, with probability at least $1-2\delta$, the optimal solution of \eqref{ine:Huber} satisfies  
\begin{align}
	\label{informal2}
	\|\hat{\vecbeta} -\vecbeta^*\|_1 \leq r_1,\quad 
	\|\hat{\vecbeta} -\vecbeta^*\|_2 \leq  r_2.
\end{align}
\end{theorem}

Suppose that the assumptions in Theorem \ref{t:main} and   $\lambda_o\sqrt{n} = 576 \sigma L^2$ hold,  we have
\begin{align}
	\label{ine:eb}
	\|\hat{\vecbeta} -\vecbeta^*\|_2 \leq 2880L^3 \sigma (c_1+c_2+c_s) (r_{d,s}+r_\delta),
\end{align}
and this coincides with \eqref{ine:bound} up to the numerical factor and the dependency of $L$.
From the conditions in Theorem \ref{t:main}, we see that Theorem \ref{t:main} does not require additional condition such that $n$ is at least greater than $s (\log (d/s))^2$, different from  \cite{SivBenRav2015Beyond} and \cite{GenKip2022Generic}.
We note that, we do not optimize the numerical factors in Theorem \ref{t:main}, and much room for improvement would remain.

The following two propositions play  important roles in the proof of Theorem \ref{t:main}:
\begin{proposition}
	\label{p:main:upper}
	Suppose that the assumptions  in Theorem \ref{t:main} hold. Let $c_1$ be a (sufficiently large) numerical constant.
	Then, for any $\vecv\in  r_1\mbb{B}^d_1 \cap r_2 \mbb{B}^d_2$, with probability at least $1-\delta$, we have 
	\begin{align}
		\label{ine:upper}
		\left| \sum_{i=1}^n \frac{\lambda_o}{\sqrt{n}} h\left(\frac{\xi_i}{\lambda_o\sqrt{n}}\right) \vecx_i^\top\vecv  \right|\leq c_1 \lambda_o \sqrt{n}L\left(\sqrt{\frac{\log(d/s)}{n}}r_1+\sqrt{\frac{\log(1/\delta)}{n}} r_2\right).
	\end{align}
\end{proposition}
As we mentioned in Section \ref{sec:rel}, in the corresponding parts to Proposition \ref{p:main:upper}, the results in \cite{SivBenRav2015Beyond} and \cite{GenKip2022Generic}  require $n$ is sufficiently large so that $n$ is at least greater than $s (\log (d/s))^2$. However, Proposition \ref{p:main:upper} does not require the quadratic dependence of $n$ on $\log(\cdot)$.
\begin{proposition}
	\label{p:main:lower}
	Suppose that the assumptions  in Theorem \ref{t:main} hold.  Let $c_2$ be a (sufficiently large) numerical constant.
	Then, for any $\vecv \in r_1\mbb{B}^d_1 \cap r_2 \mbb{B}^d_2 $, with probability at least $1-\delta$,  we have
	\begin{align}
		\label{ine:sc}
		\sum_{i=1}^n \frac{\lambda_o}{\sqrt{n}} \left(-h\left(\frac{\xi_i}{\lambda_o\sqrt{n}}-\vecx_i^\top \vecv\right)+h\left(\frac{\xi_i}{\lambda_o\sqrt{n}}\right)\right)\vecx_i^\top \vecv \geq 	\frac{\|\vecv\|_2^2}{2}-c_2L\lambda_o\sqrt{n}\left(\sqrt{\frac{\log(d/s)}{n}}r_1+\sqrt{\frac{\log(1/\delta)}{n}} r_2\right).
	\end{align}
\end{proposition}

\section*{Acknowledgement}
We would like to thank Hironori Fujisawa of the Institute of Statistical Mathematics for helpful discussions.

\appendix
\section{Proofs of the main propositions}
\label{sec:p1}
In this section, we prove Propositions \ref{p:main:upper} and \ref{p:main:lower}.
\subsection{Preparation}
\label{sec:p1-1}
In Section \ref{sec:p1-1}, we introduce Lemmas \ref{l:servtribial} and \ref{l:gc} used in the proofs of Propositions \ref{p:main:upper} and \ref{p:main:lower}. 
\begin{lemma}
	\label{l:servtribial}
	For an $L$-subexponential random vector $\vecx \in \mbb{R}^d$ with $\mbb{E} \vecx = 0$, for  any fixed vector $\vecv \in \mbb{R}^d$ and $p\geq 4$, we have
	\begin{align}
		\left(\mbb{E} |\langle \vecx^\top \vecv \rangle|^p\right)^\frac{1}{p}\leq 2Lp\left(\mbb{E} |\langle \vecx^\top \vecv \rangle|^2\right)^\frac{1}{2}.
	\end{align}
\end{lemma}
\begin{proof}
	We define $c_L = L\left(\mbb{E} |\langle \vecx^\top \vecv \rangle|^2\right)^\frac{1}{2} $.
	This proof is based on that  of Proposition 2.5.2 of \cite{Ver2018High}.
	From the definition of the $L$-subexponential random vector, we have $\mbb{E}\exp \left(|\vecx^\top \vecv|/c_L\right)\leq 2$.
	From this and Markov's inequality, for all $t\geq 0$ we have
	\begin{align}
		\label{ine:fromdefL2}
		\mbb{P}\left(|\vecx^\top \vecv|/c_L\geq t\right)&=\mbb{P}\left(e^{|\vecx^\top \vecv|/c_L}\geq e^{t}\right)\leq e^{-t} \mbb{E}e^{|\vecx^\top \vecv|/c_L} \leq 2e^{-t}.
	\end{align}
	Then, we have
	\begin{align}
		\mbb{E}|\vecx^\top \vecv|^p &= \int^\infty_0 \mbb{P}\left(|\vecx^\top \vecv|^p\geq u\right)du \stackrel{(a)}{\leq }c_L^p\int^\infty_02e^{-t}pt^{p-1}dt\stackrel{(b)}{\leq }2 p(c_Lp)^p,
	\end{align}
	where (a) follows from \eqref{ine:fromdefL2}, and (b) follows from $\Gamma(x)\leq x^x$, where $\Gamma(x)$ is the Gamma functional and Stirling's formula.
	From $p^{1/p} \leq e^{1/e}$, the proof is complete.
\end{proof}
Before introducing Lemma \ref{l:gc},  we introduce $\gamma_\alpha$-functional.
\begin{definition}[$\gamma_\alpha$-functional \cite{Dir2015Tail}]
	\label{d:tg}
	Let $(T, d)$ be a semi-metric space; that is, $d(x, z) \leq d(x, y) +
	d(y, z)$ and $d(x,y) = d(y,x)$ for $x, y, z \in T$.
	A sequence $\mc{T} = \{T_n\}_{n\geq 0}$ of subsets $T$ is called admissible if $|T_0|=1$ and $|T_n|\leq 2^{2^n}$ for all $n\geq 1$.
	For any $\alpha \in (0,\infty)$, the $\gamma_\alpha$-functional of $(T,d)$ is defined as
	\begin{align}
		\gamma_\alpha (T,d) = \inf_{\mc{T}} \sup_{t\in T}\sum_{n=0}^\infty2^{\frac{n}{\alpha}} \inf_{s\in T_n} d(t,s),
	\end{align}
	where the infimum is taken over all admissible sequences.
\end{definition}
\begin{lemma}
	\label{l:gc}
 Suppose that the assumptions on $r_1$ and $r_2$ in Theorem \ref{t:main} hold, then we have
	\begin{align}
		\gamma_1\left((r_1/r_2)^2\mbb{B}^d_0 \cap r_2 \mbb{B}^d_2,\|\cdot\|_2\right)\leq C r_1 \sqrt{s}\log (d/s),  \quad
		\gamma_2\left((r_1/r_2)^2\mbb{B}^d_0 \cap r_2 \mbb{B}^d_2,\|\cdot\|_2\right) \leq C r_1  \sqrt{\log (d/s)}.
	\end{align}
\end{lemma}
\begin{proof}
	From a standard entropy bound from chaining theory, for $\alpha = 1,2$,  we have
	\begin{align}
		\label{ine:dud}
		\gamma_\alpha\left((r_1/r_2)^2\mbb{B}^d_0 \cap r_2 \mbb{B}^d_2,\|\cdot\|_2\right) \leq C\int_0^{r_2} \left(\log N\left((r_1/r_2)^2\mbb{B}^d_0 \cap r_2 \mbb{B}^d_2,\epsilon\right)\right)^\frac{1}{\alpha}d\epsilon,
	\end{align}
	where $ N\left((r_1/r_2)^2\mbb{B}^d_0 \cap r_2 \mbb{B}^d_2,\epsilon\right)$ is the covering number of $(r_1/r_2)^2\mbb{B}^d_0 \cap r_2 \mbb{B}^d_2$, which is the minimal cardinality of an $\epsilon$-net of $(r_1/r_2)^2\mbb{B}^d_0 \cap r_2 \mbb{B}^d_2$.
	
	For  $\gamma_1(\cdot)$, from  $r_1/r_2 = 3\sqrt{s}$, $d/s\geq 3$ and Lemma 3.3 of \cite{PlaVer2013One}, we have
	\begin{align}
		\gamma_1 \left((r_1/r_2)^2\mbb{B}^d_0 \cap r_2 \mbb{B}^d_2,\|\cdot\|_2\right) \lesssim \int_0^{r_2} \log N\left((r_1/r_2)^2\mbb{B}^d_0 \cap r_2 \mbb{B}^d_2 ,\epsilon\right)d\epsilon
				  \lesssim \int_0^{r_2}9s \log \frac{r_2d}{s\epsilon} d \epsilon \lesssim r_1\sqrt{s}\log (d/s).
	\end{align}
	
	For  $\gamma_2(\cdot)$, we can repeat almost the same argument as in the case $\gamma_1(\cdot)$, so we will omit it.

\end{proof}
\subsection{Proof of Proposition \ref{p:main:upper}}
For  set $A$, we define $\mr{conv}(A)$ as a  convex hull.
First, from Lemma 3.1 of \cite{PlaVer2013One}, $r_1\mbb{B}^d_1 \cap r_2 \mbb{B}^d_2 \subset 2\mr{conv}\left((r_1/r_2)^2\mbb{B}^d_0 \cap r_2 \mbb{B}^d_2 \right)$ and Lemma D.8 of \cite{Oym2018Learning}, we have
\begin{align}
	\label{ine:p11}
	\sup_{\vecv \in r_1\mbb{B}^d_1 \cap r_2 \mbb{B}^d_2}\left|\frac{1}{n}\sum_{i=1}^nh\left(\frac{\xi_i}{\lambda_o\sqrt{n}}\right) \langle \vecx_i , \vecv\rangle \right|\leq 2\sup_{\vecv \in \frac{r_1^2}{r_2^2}\mbb{B}^d_0 \cap r_2 \mbb{B}^d_2 }\left|\frac{1}{n}\sum_{i=1}^nh\left(\frac{\xi_i}{\lambda_o\sqrt{n}}\right) \langle \vecx_i , \vecv\rangle \right|.
\end{align}

We confirm the assumptions of Corollary 5.2 of \cite{Dir2015Tail}.
	First, from Lemma \ref{l:servtribial}, we have $
		\left(\mbb{E} \left|\langle \vecx^\top \vecv \rangle\right|^p\right)^\frac{1}{p}\leq 2Lp\left(\mbb{E} \langle \vecx^\top \vecv \rangle^2\right)^\frac{1}{2} = 2Lp r_2$, and from Stirling's formula $p^p \le p! e^p$, we have $
		\mbb{E}|\langle \vecv^\top \vecx_i\rangle|^p \leq (2pL r_2)^p\leq (2eL r_2)^p p!$. Then from $-1\leq h(\cdot) \leq 1$, we have
	\begin{align}
		\label{ine:bernstein2}
	\mbb{E}\left|h\left(\frac{\xi_i}{\lambda_o\sqrt{n}}\right) \left\langle \vecx_i , \vecv \right\rangle\right|^p\leq\mbb{E}\left|\left\langle \vecx_i , \vecv \right\rangle\right|^p\leq  \frac{p!}{2} \left(4eL \sigma r_2  \right)^2(2eL r_2)^{p-2} .
	\end{align}

	Second, for any $\vecv,\vecv'\in  \frac{r_1^2}{r_2^2}\mbb{B}^d_0 \cap r_2 \mbb{B}^d_2$ and for any $i \in 1,\cdots, n$, from $-1\leq h(\cdot)\leq 1$, we have $
		\left|h\left(\frac{\xi_i}{\lambda_o\sqrt{n}}\right) \langle \vecx_i, \vecv-\vecv'\rangle\right|\leq \left|\langle \vecx_i, \vecv-\vecv'\rangle\right|$,
	and from the definition of the $L$-subexponential random vector, 
	\begin{align}
		\label{ine:orlicz}
		\left\|h\left(\frac{\xi_i}{\lambda_o\sqrt{n}}\right) \langle \vecx_i, \vecv-\vecv'\rangle\right\|_{\psi_1} \leq \left\| \langle \vecx_i, \vecv-\vecv'\rangle\right\|_{\psi_1} \leq L\|\vecv-\vecv'\|_{2}.
	\end{align}

	From \eqref{ine:p11}, \eqref{ine:bernstein2}, \eqref{ine:orlicz}, $r_\delta,r_2 \leq 1$ and  Corollary 5.2 of \cite{Dir2015Tail} 	with probability at least $1-\delta$, we have
	\begin{align}
		\sup_{\vecv \in  r_1\mbb{B}^d_1 \cap r_2 \mbb{B}^d_2 }\left|\frac{1}{n}\sum_{i=1}^nh\left(\frac{\xi_i}{\lambda_o\sqrt{n}}\right)\langle \vecx_i , \vecv\rangle \right|
		&\leq CL\left(\frac{\gamma_1\left((r_1/r_2)^2\mbb{B}^d_0 \cap r_2 \mbb{B}^d_2,\|\cdot\|_2\right)}{n}+\frac{\gamma_2\left((r_1/r_2)^2\mbb{B}^d_0 \cap r_2 \mbb{B}^d_2,\|\cdot\|_2\right)}{\sqrt{n}}+r_\delta r_2\right).
	\end{align}
	From  Lemma \ref{l:gc} and $r_{d,s}\leq 1$, with probability at least $1-\delta$, we have
	\begin{align}
		\label{ine:chaining}
		\sup_{\vecv \in  r_1\mbb{B}^d_1 \cap r_2 \mbb{B}^d_2}\left|\frac{1}{n}\sum_{i=1}^nh\left(\frac{\xi_i}{\lambda_o\sqrt{n}}\right)\langle \vecx_i , \vecv\rangle \right|&\leq CL\left(\frac{r_{d,s}}{\sqrt{s}}r_1 +r_\delta r_2\right).
	\end{align}

\subsection{Proof of Proposition \ref{p:main:lower}}
Before proceeding to the proof of  Proposition \ref{p:main:lower},
we note that  similar proofs can be seen in \cite{SunZhoFan2020Adaptive} and \cite{CheZho2020Robust}.
	From the convexity of the Huber loss, for any $a,b \in \mbb{R}$, we have $
	H(a)-H(b)\geq h(b)(a-b)$ and $H(b)-H(a) \geq h(a)(b-a)$,
	and we have
	\begin{align}
	\label{ine:positiveHuber}
	0\leq (h(a)-h(b))(a-b).
	\end{align}
	For an event $\mathfrak{E}$, let $\mr{I}_{\mathfrak{E}}$ denote the indicator function of the event $\mathfrak{E}$.
	From \eqref{ine:positiveHuber},
	we have
	\begin{align}
		\lambda_o^2\sum_{i=1}^n \left(-h\left(\frac{\xi_i}{\lambda_o\sqrt{n}}-\frac{\vecx_i^\top \vecv}{\lambda_o\sqrt{n}}\right)+h \left(\frac{\xi_i}{\lambda_o\sqrt{n}}\right) \right)\frac{\vecx_i^\top \vecv}{\lambda_o\sqrt{n}} \geq\lambda_o^2 \sum_{i=1}^n \left(-h\left(\frac{\xi_i}{\lambda_o\sqrt{n}}-\frac{\vecx_i^\top \vecv}{\lambda_o\sqrt{n}}\right)+h \left(\frac{\xi_i}{\lambda_o\sqrt{n}}\right) \right)\frac{\vecx_i^\top \vecv}{\lambda_o\sqrt{n}}\mr{I}_{E_i},
	\end{align}
	where $E_i$ denotes an  event $
		E_i := \left( \left|\frac{\xi_i}{\lambda_o\sqrt{n}}\right| \leq 1/2\right) \cap \left( \left|\frac{\vecx_i^\top \vecv}{\lambda_o\sqrt{n}}\right| \leq 1/2\right)$.
	Define the functions
	\begin{align}
	\label{def:phipsi}
		\varphi(x) =\begin{cases}
		x^2 & \mbox{ if } |x| \leq 1/4\\
		(x-1/4)^2 & \mbox{ if } 1/4\leq x \leq 1/2 \\
		(x+1/4)^2 & \mbox{ if } -1/2\leq x \leq -1/4 \\
		0 & \mbox{ if } |x| >1/2
	\end{cases} ~\mbox{ and }~
		\psi(x) = \mr{I}_{|x| \leq 1/2  }.
	\end{align}
	Let $f_i(\vecv) = \varphi\left(\frac{\vecx_i^\top \vecv}{\lambda_o\sqrt{n}}\right) \psi\left(\frac{\xi_i}{\lambda_o\sqrt{n}}\right)$
	and we have
	\begin{align}
	\label{ine:huv-conv-f}
		\sum_{i=1}^n \left(-h\left(\frac{\xi_i}{\lambda_o\sqrt{n}}-\frac{\vecx_i^\top \vecv}{\lambda_o\sqrt{n}}\right)+h \left(\frac{\xi_i}{\lambda_o\sqrt{n}}\right) \right)\frac{\vecx_i^\top \vecv}{\lambda_o\sqrt{n}}\mr{I}_{E_i}\stackrel{(a)}{\geq} \sum_{i=1}^n \varphi\left(\frac{\vecx_i^\top \vecv}{\lambda_o\sqrt{n}}\right) \psi\left(\frac{\xi_i}{\lambda_o\sqrt{n}}\right)=\sum_{i=1}^n f_i(\vecv),
	\end{align}
	where (a) follows from $\varphi(v) \leq v^2$ for $|v| \leq 1/2$. We note that
	\begin{align}
	\label{ine:f-1/4}
		f_i(\vecv) \leq\varphi\left(\frac{\vecx_i^\top \vecv}{\lambda_o\sqrt{n}}\right) \leq \min\left\{\left(\frac{\vecx_i^\top \vecv}{\lambda_o\sqrt{n}}\right)^2,\frac{1}{4}\right\}.
	\end{align}
	To bound $\sum_{i=1}^n f_i(\vecv)$ from below, we have
	\begin{align}
	\label{ine:fbelow}
	\inf_{\vecv \in r_1\mbb{B}^d_1 \cap r_2 \mbb{B}^d_2}\sum_{i=1}^n f_i(\vecv)&\geq \inf_{\vecv \in r_1\mbb{B}^d_1 \cap r_2 \mbb{B}^d_2}\mbb{E}\sum_{i=1}^nf(\vecv) -\sup_{\vecv \in r_1\mbb{B}^d_1 \cap r_2 \mbb{B}^d_2} \Big|\sum_{i=1}^n f_i(\vecv)-\mbb{E}\sum_{i=1}^n f_i(\vecv)\Big|.
	\end{align}
	Define the supremum of a random process $\Delta $ indexed by $r_1\mbb{B}^d_1 \cap r_2 \mbb{B}^d_2$ as $\sup_{ \vecv \in r_1\mbb{B}^d_1 \cap r_2 \mbb{B}^d_2} \left| \sum_{i=1}^n f_i(\vecv) - \mbb{E}\sum_{i=1}^n f_i	(\vecv) \right|$. 
	From  \eqref{def:phipsi} and  \eqref{ine:huv-conv-f}, we have
	\begin{align}
	\label{ine:aplower:tmp}
	\mbb{E}\sum_{i=1}^n f_i(\vecv)
	& \geq \frac{\|\vecv\|_2^2}{\lambda_o^2} - \sum_{i=1}^n\mbb{E}\left|\frac{\vecx_i^\top \vecv}{\lambda_o\sqrt{n}}\right|^2 \mr{I}_{\left|\frac{\vecx_i^\top \vecv}{\lambda_o\sqrt{n}}\right| \geq 1/2 }- \sum_{i=1}^n\mbb{E}\left|\frac{\vecx_i^\top \vecv}{\lambda_o\sqrt{n}}\right|^2 \mr{I}_{\left|\frac{\xi_i}{\lambda_o\sqrt{n}}\right|\geq 1/2 }.
	\end{align}
	We note that, from Lemma \ref{l:servtribial} and from the assumption on $\lambda_o$ and $320L^4 (c_1+c_2+c_2)(r_{d,s}+r_\delta) \leq 1$, $ 
		r_2 = 5 L^2 \lambda_o \sqrt{n}(c_1+c_2+c_2)(r_{d,s}+r_\delta)$, 
	\begin{align}
		\label{ine:v3}
			\mbb{E}(\vecx_i^\top \vecv)^4\leq 16 \times 81\times  L^4 \times  \|\vecv\|_2^4,\quad
		\frac{64 \times 16}{\lambda_o^4n}L^4\|\vecv\|_2^4\leq \frac{ \|\vecv\|_2^2}{4\lambda_o^2}.
	\end{align}
	We evaluate the right-hand side of \eqref{ine:aplower:tmp} for each term.
	First, for any $\vecv \in r_1\mbb{B}^d_1 \cap r_2 \mbb{B}^d_2$, we have
	\begin{align}
	\label{ap:ine:cov1}
		\sum_{i=1}^n\mbb{E}\left| \frac{\vecx_i^\top \vecv}{\lambda_o\sqrt{n}}\right|^2\mr{I}_{\left|\frac{\vecx_i^\top \vecv}{\lambda_o\sqrt{n}}\right| \geq 1/2 }
		&\stackrel{(a)}{\leq} 	\sum_{i=1}^n\sqrt{\mbb{E} \left|\frac{\vecx_i^\top \vecv}{\lambda_o\sqrt{n}} \right|^4 } \sqrt{\mbb{E}\mr{I}_{\left|\frac{\vecx_i^\top \vecv}{\lambda_o\sqrt{n}}\right| \geq 1/2 }}\stackrel{(b)}{\leq } 4	\sum_{i=1}^n\mbb{E} \left|\frac{\vecx_i^\top \vecv}{\lambda_o\sqrt{n}} \right|^4\stackrel{(c)}{\leq}	\frac{ \|\vecv\|_2^2}{4\lambda_o^2} ,
	\end{align}
		where (a) follows from H{\"o}lder's inequality, (b) follows from the relation between indicator function and expectation and  Markov's inequality, and (c) follows from \eqref{ine:v3}.
		Second, for any $\vecv \in  r_1\mbb{B}^d_1 \cap r_2 \mbb{B}^d_2$, we have
	\begin{align}
	\label{ap:ine:cov2}
		\sum_{i=1}^n\mbb{E} \left|\frac{\vecx_i^\top \vecv}{\lambda_o\sqrt{n}}\right|^2 \mr{I}_{\left|\frac{\xi_i}{\lambda_o\sqrt{n}}\right|\geq 1/2 } 
		&\stackrel{(a)}{\leq} \sum_{i=1}^n\sqrt{\mbb{E} \left|\frac{\vecx_i^\top \vecv}{\lambda_o\sqrt{n}} \right|^4} \sqrt{\mbb{E}\mr{I}_{\left|\frac{\xi_i}{\lambda_o\sqrt{n}}\right|\geq 1/2 }}\stackrel{(b)}{\leq }\sum_{i=1}^n\sqrt{\mbb{E} \left|\frac{\vecx_i^\top \vecv}{\lambda_o\sqrt{n}} \right|^4} \sqrt{\frac{4}{\lambda_o^2n} \mbb{E}\xi_i^2}\stackrel{(c)}{\leq }\frac{\|\vecv\|_2^2}{4\lambda_o^2} 	,
	\end{align}
	where (a) follows from H{\"o}lder's inequality, (b) follows from relation between indicator function and expectation and from Markov's inequality, (c) follows from \eqref{ine:v3} and the assumption on  $\xi_i$.
	Consequently, from \eqref{ap:ine:cov1} and \eqref{ap:ine:cov2} we have	
	\begin{align}
	\label{ap:h_bellow}
\frac{\|\vecv\|_2^2}{2}-\lambda_o^2\Delta\leq 	\inf_{\vecv \in r_1\mbb{B}^d_1 \cap r_2 \mbb{B}^d_2}\lambda_o^2\sum_{i=1}^n \left(-h\left(\frac{\xi_i}{\lambda_o\sqrt{n}}-\frac{\vecx_i^\top \vecv}{\lambda_o\sqrt{n}}\right) +h \left(\frac{\xi_i}{\lambda_o\sqrt{n}}\right) \right)\frac{\vecx_i^\top \vecv}{\lambda_o\sqrt{n}}.
	\end{align}

	Next, we evaluate the stochastic term $\Delta$. 
 Note that, rom \eqref{ine:f-1/4} and  \eqref{ine:v3},  we have
	$
	\mbb{E}(f_i(\vecv)-\mbb{E}f_i(\vecv))^2 \leq \mbb{E}f_i^2(\vecv) \leq\frac{L^2}{\lambda_o^2n} \|\vecv\|_2^2$.
	From this, \eqref{ine:f-1/4} and Theorem 3 of \cite{Mas2000Constants}, with probability at least $1-\delta$, we have
	\begin{align}
	\label{ine:delta}
		\Delta & \leq 2  \mbb{E} \Delta + \sqrt{\sup_{\vecv \in  r_1\mbb{B}^d_1 \cap r_2 \mbb{B}^d_2} \sum_{i=1}^n\mbb{E} (f_i(\vecv)-\mbb{E}f_i(\vecv))^2} \sqrt{8\log(1/\delta)} + 18 \log(1/\delta)\leq 2 \mbb{E} \Delta+\frac{L\sqrt{8\log(1/\delta)} }{\lambda_o} \|\vecv\|_2+ 18\log(1/\delta).
	\end{align}

	From the symmetrization inequality (Theorem 11.4 of \cite{BouLugMas2013concentration}), we have  
	\begin{align}
 \label{ine:r0}
		\mbb{E}\Delta &\leq 2   \,\mbb{E} \sup_{\vecv \in r_1\mbb{B}^d_1 \cap r_2 \mbb{B}^d_2} \left| \sum_{i=1}^n 
		a_i\varphi \left(\frac{\vecx_i^\top \vecv}{\lambda_o \sqrt{n}}\right) \psi\left(\frac{\xi_i}{\lambda_o\sqrt{n}}\right) \right|,
	\end{align} 
	where $\{a_i\}_{i=1}^n$ is a sequence of  i.i.d. Rademacher random variables independent of $\{\vecx_i,\xi_i\}_{i=1}^n$.
	We define the conditional expectation by $\{\vecx_i,\xi_i\}_{i=1}^n$ as $\mbb{E}[\, \,\cdot \mid \{\vecx_i,\xi_i\}_{i=1}^n ]$.
	From Theorem 11.6 of \cite{BouLugMas2013concentration}, we have 
	\begin{align}
	\mbb{E} \left[\sup_{\vecv \in r_1\mbb{B}^d_1 \cap r_2 \mbb{B}^d_2} \left|\sum_{i=1}^na_i\varphi \left(\frac{\vecx_i^\top \vecv}{\lambda_o \sqrt{n}}\right) \psi\left(\frac{\xi_i}{\lambda_o\sqrt{n}}\right)\right| \Bigg| \{\vecx_i,\xi_i\}_{i=1}^n \right]\leq \mbb{E} \left[\sup_{\vecv \in r_1\mbb{B}^d_1 \cap r_2 \mbb{B}^d_2} \left| \sum_{i=1}^n 
	a_i\frac{\vecx_i^\top \vecv}{\lambda_o \sqrt{n}} \right| \Bigg|  \{\vecx_i,\xi_i\}_{i=1}^n \right],
	\end{align} 
	and from the basic property of the expectation and from  the argument  similar to the proof of Proposition \ref{p:main:upper}, we have
	\begin{align}
		\label{ine:r}
		\mbb{E} \sup_{\vecv \in r_1\mbb{B}^d_1 \cap r_2 \mbb{B}^d_2} \left|\sum_{i=1}^na_i\varphi \left(\frac{\vecx_i^\top \vecv}{\lambda_o \sqrt{n}}\right) \psi\left(\frac{\xi_i}{\lambda_o\sqrt{n}}\right)\right| \leq \mbb{E} \sup_{\vecv \in r_1\mbb{B}^d_1 \cap r_2 \mbb{B}^d_2} \left| \sum_{i=1}^n 
		a_i\frac{\vecx_i^\top \vecv}{\lambda_o \sqrt{n}} \right|\leq \mbb{E}	\sup_{\vecv \in \frac{r_1^2}{r_2^2}\mbb{B}^d_0 \cap r_2 \mbb{B}^d_2 }\left|\frac{1}{n}\sum_{i=1}^n	a_i\frac{\vecx_i^\top \vecv}{\lambda_o \sqrt{n}} \right|.
		\end{align} 
	We will apply Corollary 5.2 of \cite{Dir2015Tail}, and we confirm the assumptions of Corollary 5.2 of \cite{Dir2015Tail}.
	From Lemma \ref{l:servtribial} and  Stirling's formula $p^p \le p! e^p$, we have
	\begin{align}
		\label{ine:lp2}
		\frac{1}{n}\sum_{i=1}^n \mbb{E} \left|a_i\frac{\vecx_i^\top \vecv}{\lambda_o \sqrt{n}} \right|^p \leq \frac{\mbb{E} \left|\vecx_i^\top \vecv\right|^p }{(\lambda_o \sqrt{n})^{p}}\leq \frac{2^pp^p L^p (\mbb{E}\langle \vecx_i,\vecv \rangle^2)^\frac{p}{2}}{(\lambda_o \sqrt{n})^{p}} \leq 2^p p^p \left(\frac{ L r_2}{\lambda_o \sqrt{n}}\right)^p \leq \frac{p!}{2} \left(\frac{ 4eL r_2}{\lambda_o \sqrt{n}}\right)^p.
	\end{align}
	For any $\vecv, \vecv' \in  \frac{r_1^2}{r_2^2}\mbb{B}^d_0 \cap r_2 \mbb{B}^d_2$ and for any $i=1,\cdots,n$, 
	 we have $
		\left|a_i\vecx_i^\top \vecv -a_i\vecx_i^\top \vecv'\right|\leq  \left| \left\langle \vecx_i, \vecv-	\vecv'\right\rangle \right|$ .
	Then,  from the definition of the $L$-subexponential random vector, we have
	\begin{align}
		\label{ine:psi12}
		\left\|a_i\vecx_i^\top \vecv -a_i\vecx_i^\top \vecv'\right\|_{\psi_1} &\leq L \left\|\vecv-  \vecv'\right\|_2.
	\end{align}
	From \eqref{ine:lp2} and \eqref{ine:psi12}, with arguments  similar to the proof of Proposition \ref{p:main:upper},  we have
\begin{align}
	\label{ine:chaining2}
	\frac{1}{n}\mbb{E} \sup_{\vecv \in r_1\mbb{B}^d_1 \cap r_2 \mbb{B}^d_2} \left|\sum_{i=1}^na_i\frac{\vecx_i^\top \vecv}{\lambda_o \sqrt{n}} \right| \leq C\frac{L}{\lambda_o \sqrt{n}}\left(\frac{r_{d,s}}{\sqrt{s}}r_1+r_\delta r_2\right).
\end{align}
From \eqref{ap:h_bellow}, \eqref{ine:r0}, \eqref{ine:chaining2} and $\lambda_o^2\log(1/\delta)\leq (\lambda_o \sqrt{n})^2 r_\delta^2 \leq \lambda_o \sqrt{n} r_2 r_\delta$, for any $\vecv \in r_1\mbb{B}^d_1 \cap r_2 \mbb{B}^d_2$,  we have
\begin{align}
	&\lambda_o^2\sum_{i=1}^n \left(-h\left(\frac{\xi_i}{\lambda_o\sqrt{n}}-\frac{\vecx_i^\top \vecv}{\lambda_o\sqrt{n}}\right) +h \left(\frac{\xi_i}{\lambda_o\sqrt{n}}\right) \right)\frac{\vecx_i^\top \vecv}{\lambda_o\sqrt{n}} \geq 
	\frac{\|\vecv\|_2^2}{2}-CL\lambda_o\sqrt{n}\left(\frac{r_{d,s}}{\sqrt{s}}r_1+r_\delta r_2\right).
	\end{align}

	\section{Proof of the main theorem}
	\label{asec1}
    In this section, we prove Theorem \ref{t:main}.
	Before proceeding to the proof of Theorem \ref{t:main}, we note that an idea very similar to that of Theorem \ref{t:main} can also be seen in the proof of Proposition 9.1 of \cite{AlqCotLec2019Estimation}.
	The proof is divided into two major steps.
	Assume the assumptions in Theorem \ref{t:main}, and we see that \eqref{ine:upper} and \eqref{ine:sc} hold with probability at least $1-2\delta$.
	In the remaining part of Section \ref{asec1}, we assume that \eqref{ine:upper} and \eqref{ine:sc} hold.
	\subsection{Proof of Theorem \ref{t:main}: Step1}
	\label{subsec:mainpropstep1}
	We define $\vecdelta$ as $\hat{\vecbeta}-\vecbeta^*$ and $\vecdelta_\eta$ as $\eta(\hat{\vecbeta}-\vecbeta^*)$.
	We derive a contradiction if $\|\vecdelta\|_1> r_1$ and $\|\vecdelta\|_2> r_2$ hold. Assume that $\|\vecdelta\|_1> r_1$ and $\|\vecdelta\|_2> r_2$. Then we can find some $\eta_1,\,\eta_2\in(0,1)$ such that $\|\vecdelta_{\eta_1}\|_1 = r_1$ and $\|\vecdelta_{\eta_2}\|_2 = r_2$ hold.
	Define $\eta_3 = \min\{\eta_1,\eta_2\}$.
	We consider the case $\eta_3 = \eta_2$ in Section \ref{subsubsec:mainpropstep1a}, and the case $\eta_3 = \eta_1$ in Section \ref{subsubsec:mainpropstep1b}.
	
	\subsubsection{Proof of Theorem \ref{t:main}: Step 1(a)}
	\label{subsubsec:mainpropstep1a}
	Assume that $\eta_3 = \eta_2$, and we observe  that $\|\vecdelta_{\eta_3}\|_2 = r_2$ and $\|\vecdelta_{\eta_3}\|_1 \leq r_1$ hold.
	In Section \ref{subsubsec:mainpropstep1a}, set $\eta = \eta_3(=\eta_1)$.
	Let 
	\begin{align}
	 Q'(\eta) = \frac{\lambda_o}{\sqrt{n}} \sum_{i=1}^n \left(-h\left( \frac{y_i-\vecx_i^\top (\vecbeta^*+\vecdelta_\eta)}{\lambda_o\sqrt{n}}\right)+h\left( \frac{\xi_i}{\lambda_o \sqrt{n}}\right) \right)\vecx_i^\top \vecdelta.
	 \end{align}
	From the proof of Lemma F.2. of \cite{FanLiuSunZha2018Lamm}, we have $\eta Q'(\eta) \leq \eta Q'(1)$, which means that
	\begin{align}
	\label{ine:det:1}
	&\sum_{i=1}^n \frac{\lambda_o}{\sqrt{n}}\left(-h\left( \frac{y_i-\vecx_i^\top (\vecbeta^*+\vecdelta_\eta)}{\lambda_o\sqrt{n}}\right)+h\left( \frac{\xi_i}{\lambda_o \sqrt{n}}\right)\right) \vecX_i^\top \vecdelta_\eta \leq \sum_{i=1}^n \frac{\lambda_o}{\sqrt{n}}\eta \left(-h  \left(\frac{\vecx_i^\top \hat{\vecbeta}}{\lambda_o\sqrt{n}}\right)+h\left( \frac{\xi_i}{\lambda_o \sqrt{n}}\right)\right)\vecX_i^\top\vecdelta.
	\end{align}
	For a vector $\vecv \in \mbb{R}$, let $\partial \vecv$ be the sub-differential of $\|\vecv\|_1$, and
	note that $\|\hat{\vecbeta}\|_1-\|\vecbeta^*\|_1 \leq \langle \partial \hat{\vecbeta},\vecdelta\rangle$, which is  the definition of the sub-differential.
	Therefore, 
	adding $\eta \lambda_s(\|\hat{\vecbeta}\|_1-\|\vecbeta^*\|_1) $ to both sides of \eqref{ine:det:1}, we have
	\begin{align}
		\label{ine:det2:1}
		&\sum_{i=1}^n \frac{\lambda_o}{\sqrt{n}}\left(-h\left( \frac{y_i-\vecx_i^\top (\vecbeta^*+\vecdelta_\eta)}{\lambda_o\sqrt{n}}\right) +h\left( \frac{\xi_i}{\lambda_o \sqrt{n}}\right)\right) \vecX_i^\top \vecdelta_\eta \leq \sum_{i=1}^n \frac{\lambda_o}{\sqrt{n}} h\left( \frac{\xi_i}{\lambda_o \sqrt{n}}\right) \vecX_i^\top \vecdelta_\eta+\eta \lambda_s(\|\vecbeta^*\|_1-\|\hat{\vecbeta}\|_1),
	\end{align}
	where we use the optimality of $\hat{\vecbeta}$.
		We note that $
			\eta \lambda_s(\|\vecbeta^*\|_1-\|\hat{\vecbeta}\|_1) \leq \lambda_s \|\vecdelta_\eta\|_1$, and from \eqref{ine:upper}, \eqref{ine:sc} and $r_1 = 3\sqrt{s}r_2$, we have $
			\|\vecdelta_\eta\|_2^2/4 \leq (c_1+c_2+c_s) \lambda_o \sqrt{n}L\left(r_{d,s}+ r_\delta\right)\|\vecdelta_2\|_2$.
		From this and $\|\vecdelta_\eta\|_2\geq 0$, we have $\|\vecdelta_\eta\|_2 \leq 4(c_1+c_2+c_s) \lambda_o \sqrt{n}L\left(r_{d,s}+ r_\delta\right)$.
		This is a contradiction because $
		\|\vecdelta_\eta\|_2 = r_2$.
	
	\subsubsection{Proof of Theorem \ref{t:main}: Step 1(b)}
	\label{subsubsec:mainpropstep1b}
	Assume that $\eta_3 = \eta_1$, then we see that $\|\vecdelta_{\eta_3}\|_2 \leq r_2$ and $\|\vecdelta_{\eta_3}\|_1 = r_1$ hold. 
	In Section \ref{subsubsec:mainpropstep1b}, set $\eta = \eta_3(=\eta_1)$.
	First assume  $\|\vecdelta_\eta \|_2 \geq \|\vecdelta_\eta\|_1/\sqrt{s}$.
	Then, we immediately observe that $\|\vecdelta_\eta \|_1 \leq \sqrt{s}\|\vecdelta_\eta\|_2$ holds, and
	from $\|\vecdelta_\eta \|_1 = r_1 = 3\sqrt{s}r_2$, which is a contradiction.
	
	Next, assume  $\|\vecdelta_\eta \|_2 \leq \|\vecdelta_\eta\|_1/\sqrt{s}$.
	From \eqref{ine:positiveHuber}, we can see that the lef-hand side  of \eqref{ine:det:1} is positive, and we have
		\begin{align}
			\label{ine:det:5}
		0\leq \frac{1}{n}\sum_{i=1}^n\lambda_o\sqrt{n}  h\left( \frac{y_i-\vecx_i^\top \vecbeta^*}{\lambda_o \sqrt{n}}\right)\vecX_i^\top\vecdelta_\eta+	\eta \lambda_s(\|\vecbeta^*\|_1-\|\hat{\vecbeta}\|_1).
		\end{align}
		From \eqref{ine:det:5}, \eqref{ine:upper}, $r_1 = 3\sqrt{s}r_2$, $\max\{r_2, r_\delta, r_{d,s}\}\leq 1$ and the assumption $\|\vecdelta_\eta \|_2 \leq \|\vecdelta_\eta\|_1/\sqrt{s}$, we have
		\begin{align}
		0 &\leq c_1 \lambda_o \sqrt{n}L\left(r_{d,s}r_1/\sqrt{s} + r_\delta r_2\right)+\eta \lambda_s(\|\vecbeta^*\|_1-\|\hat{\vecbeta}\|_1)\leq  C_s\|\vecdelta_\eta\|_1+\eta \lambda_s (\|\vecbeta^*\|_1-\|\hat{\vecbeta}\|_1),
		\end{align}
		where $C_s = c_1 \lambda_o \sqrt{n}L(r_{d,s}+  r_\delta)/\sqrt{s}$.
		Define $\mc{J}_{\veca}$ as the index set of the non-zero entries of $\veca$, and $\vecdelta_{\eta, \mc{J}_{\veca}}$ as a vector such that $(\vecdelta_{\eta, \mc{J}_{\veca}})_i=(\vecdelta_\eta)_i$ for $i \in \mc{J}_{\veca}$ and $(\vecdelta_{\eta, \mc{J}_{\veca}})_i = 0$ for $i \notin \mc{J}_{\veca}$.
		Furthermore, we see
		\begin{align}
		0
		&\leq C_s\|\vecdelta_\eta\|_1+\eta \lambda_s(\|\vecbeta^*\|_1-\|\hat{\vecbeta}\|_1) \nonumber \\
		& \leq C_s (\|\vecdelta_{\eta,\mc{J}_{\vecbeta^*}}\|_1+\|\vecdelta_{\eta,\mc{J}^c_{\vecbeta^*}}\|_1) + \eta\lambda_s(\|\vecbeta^*_{\mc{J}_{\vecbeta^*}}- \hat{\vecbeta}_{\mc{J}_{\vecbeta^*}}\|_1-\|\hat{\vecbeta}_{\mc{J}^c_{\vecbeta^*}}\|_1)=\left(\lambda_s + C_s\right)\|\vecdelta_{\eta,\mc{J}_{\vecbeta^*}}\|_1
		+\left(-\lambda_s+ C_s\right)\|\vecdelta_{\eta,\mc{J}^c_{\vecbeta^*}}\|_1.
		\end{align}
		Subsequently, from $c_s\geq 5c_1$, we have
			\begin{align}
				\label{ine:preRE}
				\| \vecdelta_{\eta,\mc{J}_{\vecbeta^*}^c}\|_1 \leq \frac{\lambda_s + C_s}{\lambda_s - C_s }\| \vecdelta_{\eta,\mc{J}_{\vecbeta^*}}\|_1\leq (3/2) \| \vecdelta_{\eta,\mc{J}_{\vecbeta^*}}\|_1,
			\end{align}
			and we have $
				\|\vecdelta_\eta\|_1 = \|\vecdelta_{\eta,\mc{J}_{\vecbeta^*}}\|_1+ \|\vecdelta_{\eta,\mc{J}^c_{\vecbeta^*}}\|_1 \leq (3/2+1) \|\vecdelta_{\eta,\mc{J}_{\vecbeta^*}}\|_1\leq (5/2)\sqrt{s} \|\vecdelta_\eta\|_2$.
			 Then, there is  a contradiction because $\|\vecdelta_\eta\|_1 =r_1 = 3\sqrt{s}r_2$.

	\subsection{Proof of Theorem \ref{t:main}: Step 2}
	\label{asubsec:2}
	From the arguments in Section \ref{subsec:mainpropstep1}, we have $\|\vecdelta\|_2\leq r_2$ or $\|\vecdelta\|_1\leq r_1$ holds.
	In Section \ref{asubsubsec:2-1}, assume that $\|\vecdelta\|_2 > r_2$ and $\|\vecdelta\|_1 \leq r_1$ 
	hold and then derive a contradiction.
	In Section \ref{asubsubsec:2-2}, assume that $\|\vecdelta\|_2 \leq r_2$ and $\|\vecdelta\|_1 > r_1$ 
	hold and then derive a contradiction.
	Finally, we have $
		\|\hat{\vecbeta}-\vecbeta^*\|_2 \leq r_2$ and $\|\hat{\vecbeta}-\vecbeta^*\|_1 \leq r_1$,
	and the proof is complete.
	\subsubsection{Proof of Theorem \ref{t:main}: Step 2(a)}
	\label{asubsubsec:2-1}
	Assume that $\|\vecdelta\|_2 > r_2$ and $\|\vecdelta\|_1 \leq r_1$ hold, and
	then we  observe $\eta_4\in (0,1)$ such that $\|\vecdelta_{\eta_4}\|_2=r_2$ holds.
	Note that $\|\vecdelta_{\eta_4}\|_1\leq r_1$ also holds.
	Then, from the same arguments of Section \ref{subsubsec:mainpropstep1a}, we have  a contradiction.
	
	\subsubsection{Proof of Theorem \ref{t:main}: Step 2(b)}
	\label{asubsubsec:2-2}
	Assume that $\|\vecdelta\|_2 \leq r_2$ and $\|\vecdelta\|_1 > r_1$ hold, and
	then we observe $\eta_5\in (0,1)$ such that $\|\vecdelta_{\eta_5}\|_1=r_1$ holds.
	We note that $\|\vecdelta_{\eta_5}\|_2\leq r_2$ also holds.
	Then, from the same arguments of Section \ref{subsubsec:mainpropstep1b}, we have  a contradiction.

\end{document}